\documentclass[10pt]{amsart}

\usepackage{amsmath}
\usepackage{amsthm}
\usepackage{amsfonts}
\usepackage{amssymb,latexsym}
\usepackage[all]{xy}
\usepackage{graphicx}
\usepackage[mathscr]{eucal}
\usepackage{verbatim}
\usepackage{hyperref}

\theoremstyle{plain}
    \newtheorem{thm}{Theorem}[section]

    \newtheorem{prop}[thm]{Proposition}
    \newtheorem{lemma}[thm]{Lemma}
    
    \newtheorem{cor}[thm]{Corollary}

\theoremstyle{definition}
    \newtheorem{defn}[thm]{Definition}
    
\theoremstyle{remark}
    \newtheorem{rem}[thm]{Remark}
    \newtheorem{example}[thm]{Example}

\numberwithin{equation}{section}

\newcommand{\rar}{\ensuremath{\rightarrow}}

\newcommand{\Hom}{\textup{Hom}}

\newcommand{\stmod}{\textup{stmod}}

\newcommand{\uHom}{\underline{\Hom}}

\newcommand{\Ind}{\textup{Ind}}
\newcommand{\Res}{\textup{Res}}
\newcommand{\uInd}{\underline{\Ind}}
\newcommand{\uRes}{\underline{\Res}}

\newcommand{\down}{_{\downarrow P}}

\def\HHH{\operatorname{H}\nolimits}
\def\HHHH{\operatorname{\widehat{H}}\nolimits}
\def\Hom{\operatorname{Hom}\nolimits}

\begin{document}

\title{Ghosts and Strong Ghosts in the stable module category}

\date{\today}

\author[Jon F. Carlson]{Jon F. Carlson}
\address{Department of Mathematics \\
University of Georgia \\
Athens, GA 30602, USA}
\email{jfc@math.uga.edu}

\author{Sunil K. Chebolu}
\address{Department of Mathematics \\
Illinois State University \\
Normal, IL 61790 USA}
\email{schebol@ilstu.edu}

\author{J\'{a}n Min\'{a}\v{c}}
\address{Department of Mathematics\\
University of Western Ontario\\
London, ON N6A 5B7, Canada}
\email{minac@uwo.ca}

\thanks{The first author was partially supported by a grant
  from the  NSF (DMS-1001102) and a grant from NSA 
 (H98230-15-1-0007), the second author by a grant
  from the NSA (H98230-13-1-0238) and the
third  author is supported from NSERC grant (R0370A01)}



\begin{abstract}
  Suppose that $G$ is a finite group and $k$ is a field of characteristic
  $p>0$. A \emph{ghost map} is a map in the stable category of
  finitely generated $kG$-modules which induces the zero map
  in Tate cohomology in all degrees. In an earlier paper we showed that the 
  thick subcategory generated by the trivial module
has no nonzero ghost maps if and only if
  the Sylow $p$-subgroup of $G$ is cyclic of order 2 or 3. 
  In this paper we introduce  and study  variations of ghost maps. 
  In particular,  we consider the behavior of ghost maps under restriction
  and induction functors. We find all groups satisfying a strong form
  of Freyd's generating hypothesis and show that ghosts can
  be detected on a finite range of degrees of Tate cohomology. We also
  consider maps which mimic ghosts in high degrees. 
\end{abstract}

\maketitle
\thispagestyle{empty}

\section{Introduction} \label{intro}
Suppose that $G$ is a finite group and $k$ is a field whose characteristic
divides the order of $G$. A ghost map is a map  between  $kG$-modules 
that induces the zero map in Tate cohomology in all degrees.
There is an extensive literature on ghost maps in the stable
module category  \cite{CCM, CCM2, CCM3, CCM4, CarCheMin, CarCheMin2,
CW, CW2, Aksu-Green}  and in other triangulated categories
\cite{lock, hl1, hl2, hl3}.  Most of this literature
was inspired  by a famous conjecture in homotopy theory due to
Peter Freyd \cite{Freyd} from 1965 which goes under the name of
the generating hypothesis.  This conjecture asserts that there are
no nontrivial ghost maps in the category of finite  spectra.
In the category of spectra a ghost map is a map
which induces the zero map in stable homotopy in all degrees.
Although not much progress has been made on this conjecture,  analogues of ghost maps and the generating hypothesis
have been introduced and studied in other  triangulated categories
in the aforementioned papers.

Motivated by the above work, 
we introduce and study some variations of ghosts maps in
the stable module category of a modular group algebra.
Our analysis includes a complete characterization of the finite groups where a strong version of the generating hypothesis  holds.  The relevant definitions are as follows.  Throughout
the paper we assume that $G$ is a finite group and that $k$ is a field
of characteristic $p$ dividing the order of $G$.

\begin{defn}  \label{defn} Let $M$ and $N$ be finitely
  generated $kG$-modules and let $\varphi \colon M \rar N$
  be a $kG$-homomorphism.  
We say that $\varphi$ is a \emph{ghost}  if it induces the
  zero map in Tate cohomology in all degrees. That is,  for all $i$, the induced map
  \[
  \xymatrix{
 \varphi_* \colon\HHHH^i(G, M) \ar[r] & \HHHH^i(G, N)
}
  \]
is the zero map. 
The map $\varphi$ is a \emph{strong ghost}  if it is a ghost
\and remains a ghost on restriction to all subgroups $H$ of $G$.  That is,  for all $i$ and all subgroups $H$ of $G$, the induced map
  \[
  \xymatrix{
 \varphi_* \colon \HHHH^i(H, M_{\downarrow H}) \ar[r] & 
\HHHH^i(H, N_{\downarrow H})
}
  \]
is the zero map. 

The map $\varphi$ is an \emph{eventual ghost} if it induces the zero map
  in Tate cohomology in all sufficiently large degrees. That is,
  $\varphi$ is an eventual ghost provided there is an integer $n$ such that 
  \[
\xymatrix{
  \varphi_* \colon \ \HHHH^i(G, M) \ar[r] &  \HHHH^i(G, N) 
}
\]
is zero for all integers $i > n$.
\end{defn}

In \cite{CarCheMin2} and \cite{CCM4}  it was shown that every ghost map between $kG$-modules
in the thick subcategory of the stable category generated by the trivial
module is zero if and only if the Sylow $p$-subgroup 
of $G$ is $C_2$ or $C_3$. This settled
Freyd's generating hypothesis for modular group algebras. As for strong
ghosts, there is only one more case in which all strong ghosts vanish. 
A main result of this paper proves that
every strong ghost is zero if and only if the Sylow $p$-subgroup
of $G$ is cyclic of order 2, 3 or 4.  The result can be viewed as saying that
a strong form of Freyd's generating hypothesis holds only for the
groups mentioned. The theorem was used in the work \cite{CW2}, which
cited an early version of this manuscript. The proof of the theorem
on strong ghosts is constructive, using
Auslander and Reiten's theory of almost split sequences together with
standard induction and restriction methods. These results are proved in
Sections 2 and 3.

In Section 4, we demonstrate that the property of being a ghost
for a map $\varphi: M \to N$ is detected in a finite range of cohomology
degrees, which depend on $M$ and $N$.  An important step in the development
is a proof that the dual of any ghost is again a ghost. This also
applies to strong ghosts.

In the final section of this paper we study eventual ghosts.
It is clear that every ghost map is also an eventual ghost, so
the converse seems to be a natural question. The answer is that this
happens if and only if $G$ has periodic cohomology. 
This question is related to the finite generation of Tate cohomology
studied in \cite{CarCheMin}.
\vskip.1in
\noindent
\textbf{Acknowledgements:}
We would like to thank  Mark Hovey,  Jonathan Pakianathan and 
Gaohong Wang for helpful conversations and questions. In particular,
we thank Dan Christensen and Wang for pointing out misprints and an 
error in an earlier manuscript.  We also wish to thank the referee for many 
suggestions which improved the exposition of the paper.

\section{Preliminaries: ghosts under restriction, induction and duality}

Throughout the paper we let $G$ be a finite group and let $k$ be
a field of characteristic $p$. Recall that $kG$ is a self-injective
algebra, meaning that projective modules are injective and {\it vice
  versa}. The modules that we consider are all finitely generated.
If $M$ is a $kG$-module and $\varphi: P \to M$ is a projective cover,
then the kernel of $\varphi$ is denoted $\Omega(M)$. Dually, if
$\theta: M \to Q$ is the injective hull of $M$, then the cokernel
of $\theta$ is denoted $\Omega^{-1}(M)$. Inductively, we write
$\Omega^n(M) = \Omega(\Omega^{n-1}(M))$ and $\Omega^{-n}(M)
= \Omega^{-1}(\Omega^{1-n}(M))$. 

Most of the  objects of this study reside in
the stable module category $\stmod(kG)$. It is the
category whose objects are finitely generated
left $kG$-modules. The set of morphisms between $kG$-modules $M$ and
$N$ in $\stmod(kG)$ is denoted $\uHom_{kG}(M,N)$. It is the
quotient of the $k$-vector space of $kG$-module homomorphisms 
by the subspace of those maps
that factor through a projective module.  Thus, projective modules
are zero in this category. A stably trivial map is a map between
$kG$-modules which factors through a projective.  The stable module
category is a triangulated category in which the triangles come from
short exact sequences of $kG$-modules. The translation functor
is $\Omega^{-1}$. The stable module category is the
natural home for Tate cohomology. A fact that we use often is that,
for $M$  a $kG$-module, the Tate cohomology group
$\HHHH^i(G, M)$ is isomorphic to $\uHom_{kG}(\Omega^i(k) , M)$.
For more details on group cohomology and the stable category,
see \cite{carlson-modulesandgroupalgebras, CTVZ}.

Let $H$ be a subgroup of a group $G$. There are restriction and
induction functors  between the corresponding stable categories:
\[
\uRes_{G,H} \colon \stmod(kG) \rar \stmod(kH)
\] 
which remembers only the action of $H$ on a module $M$ and 
\[
\uInd_H^G \colon \stmod(kH) \rar \stmod(kG)
\]
which takes a $kH$-module $M$ to $M^{\uparrow G} =
kG \otimes_{kH} M$. We denote
the restriction of $M$ to $H$ by $M_{\downarrow H}$ or just $M_H$.
The Eckmann-Shapiro Lemma \cite{carlson-modulesandgroupalgebras}  
says that these two functors are adjoint
to each other.  In particular, for $M$ a $kH$-module, 
\[
\HHHH^i(H, M) \cong \HHHH^i(G, M^{\uparrow G})
\]
for all $i$.  

We now develop some tools using the induction and restriction functors 
are used in the proof of Theorem \ref{thm:strongghost}.

\begin{lemma} \label{faithful}
Let $P$ be a Sylow $p$-subgroup of $G$. The restriction functor 
$\uRes_{G,P}$ is faithful. That is, if  $\varphi \colon M \rar N$  is
a map of $kG$-modules such that $\uRes_{G,P}(\varphi): M_{\downarrow P} 
\to N_{\downarrow P}$ is
zero in $\stmod(kP)$, then $\varphi$ is zero in $\stmod(kG)$.
\end{lemma}

\begin{proof}
  Suppose that $\Res_{G,P}(\varphi)$ factors through a
  projective $kP$-module $T$:
\[ 
\xymatrix{
M\down \ar[r]^\beta & T \ar[r]^\gamma & N\down
}
\]
Consider the diagram of $kG$-modules
\[ 
\xymatrix{
M \ar[r]^{\widehat{\beta}} & T^{\uparrow G} \ar[r]^{\widehat{\gamma}} & N,
}
\]
where $\widehat{\beta}$ and $\widehat{\gamma}$ are the adjoints of the maps
$\beta$ and $\gamma$. Let $n$ be the index of $P$ in $G$. It is
easy to verify that $\widehat{\gamma} \widehat{\beta} = n \varphi$.
Since $P$ is a Sylow $p$-subgroup, the integer $n$ is
coprime to $p$, and  therefore it is invertible in $k$.
Replacing $\widehat{\beta}$ by $(1/n)\widehat{\beta}$, we get a factorization
of $\varphi$ through $T^{\uparrow G}$, a projective $kG$-module. This
means $\varphi$ is zero in $\stmod(kG)$.
\end{proof}

\begin{cor}
A map $\varphi: M \to N$ is a ghost whenever its restriction to a Sylow $p$-subgroup
of $G$ is a ghost. 
\end{cor}

\begin{proof}
Let $\varphi: M \to N$ be a map such that its restriction to a Sylow $p$-subgroup of $G$ is a ghost. To show that $\varphi$ is ghost, we have to show that the composition 
\[ 
\xymatrix{
\Omega^i(k) \ar[r]^f & M \ar[r]^\varphi & N 
}
\]
is zero in the stable module category for all integers $i$ and all $f$. Since restriction to a Sylow $p$-subgroup $P$ was shown to be faithful, it is enough to show that restriction of this composition to $P$ is zero. But that latter is true because $\varphi$ restricted to $P$ is a ghost by assumption.
\end{proof}

\begin{prop} \label{strong-restrict}
  A map $\varphi: M \to N$ is a strong ghost if and only if it is a
  ghost on restriction to every $p$-subgroup of $G$.
\end{prop}

\begin{proof}
  The ``only if'' part is obvious from the definition. So suppose that
  $\uRes_{G,Q}(\varphi)$ is a ghost for every $p$-subgroup $Q$ of $G$.
  Suppose that $H$ is any subgroup of $G$ and that $Q$ is a Sylow
  $p$-subgroup of $H$. Since $\uRes_{G,Q}(\varphi) =
    \uRes_{H,Q}(\uRes_{G,H}(\varphi))$ is a ghost, 
    $\uRes_{G,H}(\varphi)$ is a ghost by the last lemma.
    So $\varphi$ is a strong ghost.
\end{proof}

\begin{prop} \label{rest-strong}
  Suppose that $H$ is a subgroup of $G$ that contains a Sylow $p$-subgroup
of $G$.
  Let $\varphi: M \to N$ be a $kG$-homomorphism. Then $\varphi$ is a strong
  ghost if and only if $\uRes_{G,H}(\varphi)$ is a strong ghost.
\end{prop}

\begin{proof}
  The ``only if'' part is obvious from the definition. Suppose that
  $\uRes_{G,H}(\varphi)$ is a strong ghost. Any $p$-subgroup $Q$ of $G$
  is conjugate to a subgroup of $H$ and hence $\uRes_{G,Q}(\varphi)$
  is a ghost by the last proposition. Therefore, again by the last proposition,
  $\varphi$ is a strong ghost. 
\end{proof}
 
For the induction functor we get an even stronger result. We 
also need this in the proof of Theorem \ref{thm:strongghost}.

\begin{prop} \label{strong-induct}
  Suppose that $H$ is a subgroup of $G$ and that $\varphi: M \to N$ is a
  homomorphism of $kH$-modules. Then $\varphi$ is a strong ghost if and only
  if $\uInd_H^G(\varphi)$ is a strong ghost. 
\end{prop}

\begin{proof}
  Let $Q$ be a $p$-subgroup of $G$. Then by the Mackey decomposition theorem,
  \[
  (M^{\uparrow G})_{\downarrow Q}  \cong \bigoplus_{x \in Q\backslash G / H}
  ((x \otimes M)_{\downarrow Q \cap xHx^{-1}})^{\uparrow Q}
  \]
  where the sum is over a set of representatives of the $Q$-$H$ double
  cosets. Note that for $x \in G$ and $m \in M$, the map
  $\uInd_H^G(\varphi)$ on $M^{\uparrow G}$ is given by $\varphi(x\otimes m)
  = x \otimes \varphi(m)$. The point of this observation is that
  $\uInd_H^G(\varphi)$ commutes with the Mackey decomposition. Hence
  $\uRes_{G,Q}(\uInd_H^G(\varphi))$ is a direct sum of maps
  \[
\xymatrix{
  \varphi_x : \quad ((x \otimes M)_{\downarrow Q \cap xHx^{-1}})^{\uparrow Q} 
\ar[r] &  ((x \otimes N)_{\downarrow Q \cap xHx^{-1}})^{\uparrow Q}
}
  \]
  where, again, the sum is taken over a set of representative of the
  $Q$-$H$-double cosets. It follows that
  $\uRes_{G,Q}(\uInd_H^G(\varphi))$ is a ghost
  if and only if every $\varphi_x$ is a ghost.

  Suppose that $\uInd_H^G(\varphi)$ is a strong ghost. If $Q$ is a
  $p$-subgroup of $H$, then $\uRes_{G,Q}(\uInd_H^G(\varphi))$ is a ghost
  and $\varphi_x$ is a ghost for every $x$. In the case that $x=1$,
  we have that $\varphi_1 = \uRes_{H,Q}(\varphi)$ which
  is a ghost.  Proposition  \ref{rest-strong} implies that  
$\varphi$ is a strong ghost.

On the other hand, if $\varphi$ is a strong ghost, then 
for any $p$-subgroup $Q$ of $G$, we have that
\[
\xymatrix{
\varphi_x = \uRes_{xHx^{-1},Q \cap xHx^{-1}}(x \otimes \varphi) : 
(x \otimes M)_{\downarrow Q \cap xHx^{-1}} \ar[r] & 
(x \otimes N)_{\downarrow Q \cap xHx^{-1}}
}
\]
is a ghost. Then the naturality of the Eckmann-Shapiro isomorphism
\[
\HHHH^*(Q, ((x \otimes M)_{\downarrow Q \cap xHx^{-1}})^{\uparrow Q}) \cong
\HHHH^*(Q \cap xHx^{-1},  (x \otimes M)_{\downarrow Q \cap xHx^{-1}})
\]
  asserts that each $\varphi_x$ is a ghost and hence
  $\uInd_H^G(\varphi)$ is a strong ghost by Proposition  \ref{rest-strong}.
\end{proof}

The last two propositions give us the following corollary which is 
useful in Theorem \ \ref{thm:strongghost}.

\begin{cor}
Let $P$ be a Sylow $p$-subgroup of a group $G$. The strong generating hypothesis holds for $\stmod(kG)$ if and only if it holds for $\stmod(kP)$.
\end{cor}

\section{Groups with no strong ghosts}  \label{sec:nostrong}
In this section we consider groups whose stable module categories have no
strong ghosts.  Recall, from Theorem 1.1 of
\cite{CarCheMin2}, that the thick subcategory of $\stmod(kG)$ 
generated by the trivial module has no nontrivial
ghosts if and only if the Sylow $p$-subgroup of $G$ is $C_2$ or $C_3$. 
If the Sylow $p$-subgroup of $G$ is either $C_2$
with $p=2$ or $C_3$ with $p=3$,
then every ghost is a strong ghost, and hence there are no nontrivial strong
ghosts in $\stmod(kG)$. So we consider $C_4$.

\begin{prop} \label{cyclic4}
Suppose that $k$ is a field of characteristic $2$. Then
$\stmod(kC_4)$ has no nontrivial strong ghosts. 
\end{prop}

\begin{proof} 
Let $G \cong C_4$ a cyclic group of order 4. The group 
algebra has exactly three isomorphism classes of indecomposable
modules represented by modules $M_i$ of dimension $i$ for
$i = 1, 2, 3$. Observe that $M_1 \cong k$, $M_2 \cong k_H^{\uparrow G}$ and
$M_3 = \Omega(k)$. Here $H$ is the subgroup of $G$ of order 2 and $k_H$ denotes the trivial $kH$-module.
All three of these modules are self-dual.
Clearly, no nonzero (in the stable category) map from $M_i$ to $M_j$
can be ghost if either $i$ or $j$ is 1 or 3. This follows from
the definition and Proposition \ref{prop:dualghost}. Consequently,
any possible nonzero strong ghost maps $M_2$ to itself. 
However, the restriction
of $M_2$ to $H$ is $(M_2)_{\downarrow H} \cong k_H \oplus k_H$. 
Because any strong
ghost from $M_2$ to itself induces the zero map on 
$\HHHH^0(H, (M_2)_{\downarrow H})$,
it is actually the zero map.
\end{proof}

Now notice in the cases examined thus far,  whenever $G$ is a $p$-group and 
$\stmod(kG)$ has no strong ghosts,  the only
indecomposable modules are either syzygies of the 
 trivial module
or induced modules from proper subgroups. This, in fact, is the whole story.

\begin{prop} \label{almostsplitlemma} Let $G$ be a finite group. 
If there exists an indecomposable nonprojective $kG$-module $M$ such that 
\begin{enumerate}
\item for every nontrivial $p$-subgroup $Q$ of $G$, $Q$ not a Sylow
  $p$-subgroup, the module $M$ is not a
  direct summand of a module induced from $Q$,  and
\item $M \ncong \Omega^i(k)$ for any $i$,
\end{enumerate}
then there exists a nontrivial strong ghost in $\stmod(kG)$.
\end{prop}

\begin{proof}
Consider the almost split sequence which ends in $M$ \cite{Aus-Car}:
\[
\xymatrix{
0 \ar[r] & \Omega^2(M) \ar[r] & X \ar[r] & M \ar[r] & 0
}
\]
This sequence is represented by a map $\varphi \colon M \rar \Omega(M)$ in the stable module category. We claim that $\varphi$ is a strong ghost in $\stmod(kG)$.  From the second condition on $M$, we know that $\varphi$
is a nontrivial ghost in $\stmod(kG)$.  See \cite{CarCheMin2}.

The first condition on $M$ implies that
the above sequence splits on restriction
to any $p$-subgroup $Q$. Let $\theta: (M_{\downarrow Q})^{\uparrow G} \to M$ be a homomorphism. This map cannot be a split epimorphism by the first condition.  By the
definition of an almost split sequence, then there is a map $\mu$ such
that the diagram
\[
\xymatrix{
&&& (M_{\downarrow Q})^{\uparrow G} \ar[d]^\theta \ar[dl]_\mu \\
0 \ar[r] &  \Omega^2(M) \ar[r] & X   \ar[r] & M  \ar[r]  & 0
}
\]
commutes. Hence, the map $\uHom_{kG}((M_{\downarrow Q})^{\uparrow G}, X) 
\to \uHom_{kG}((M_{\downarrow Q})^{\uparrow G}, M)$
is surjective. However, then the Eckmann-Shapiro Lemma tells us that 
$\uHom_{kQ}(M_{\downarrow Q}, X_{\downarrow Q}) 
\to \uHom_{kQ}(M_{\downarrow Q}, M_{\downarrow Q})$ is surjective and the
almost split sequence splits on restriction to $Q$.
This means that $\Res_{G,Q}(\varphi) = 0$, and by Proposition
 \ref{strong-restrict}, $\varphi$ is a strong ghost. 
\end{proof}

We are now prepared to prove the main theorem of this section.

\begin{thm} \label{thm:strongghost}
  Let $G$ be a finite group and $k$ be a field of characteristic $p$.
  In the stable module category $\stmod(kG)$
  every strong ghost is zero if and only if the Sylow $p$-subgroup
  of $G$ is $C_2$, $C_3$ or $C_4$.
\end{thm}

\begin{proof}
  By Lemma \ref{faithful}, Propositions \ref{rest-strong} and \ref{strong-induct}, we may assume that $G$ is a $p$-group.
  The ``if'' part is a consequence of Theorem 1.1 of \cite{CarCheMin2}
  and Proposition \ref{cyclic4}.  By Proposition \ref{almostsplitlemma},
  it remains only to show that if $G$ is a $p$-group that is not cyclic
  of order 2, 3, or 4, then $G$ has an indecomposable nonprojective module $M$ that is
  not a syzygy of the trivial module and not a direct summand of a module
  induced from a proper nontrivial subgroup of $G$.  A $p$-group that 
is not cyclic of order $2, 3$ or $4$ belongs to exactly one of the 
following 3 disjoint cases. In each of these cases, we  
show that there exists a module $M$ with the above-mentioned 
properties.  We use the fact that direct summands of modules
  induced from proper subgroups have dimension divisible by $p$.

Assume that $G$ is cyclic of order at least $5$.  In this case we let  $M$ be 
any indecomposable module of dimension $n$, where $n$ is not 1 or 
$\vert G\vert-1$ and not divisible by $p$.  More specifically, we 
take $n = 2$ when $p$ is odd and $n= 3$ when $p =2$. Since $G$ is 
cyclic and has order at least 5, the  unique indecomposable module 
of this dimension has the desired properties. 
  
Next, assume that  $G$ is not cyclic and has order at least $5$. 
Consider a composition series
\[ 0 
\subseteq A_1 (= k) \subseteq A_2 \subseteq \cdots \subseteq A_s (= kG) 
\]
of submodules of $kG$ such that each successive quotient is isomorphic to 
$k$.  Such a series exits because $G$ being a $p$-group has only one 
simple module equal to $k$.    Observe that 
since $G$ is not cyclic $s > p+1$. Let $N = A_{p+1}$ and
let $M = kG/N$. Note that $M^G = k$ and hence $M$ is indecomposable, 
and moreover $M$ has dimension not a multiple of $p$.  We claim that $M$  
is  not a syzygy of the trivial module. To see this first note that 
a syzygy of the trivial module has dimension $\pm{1}$ modulo $|G|$, 
whereas the module $M$ has dimension 
$-(p+1)$ modulo $|G|$. So if $M$ is a syzygy of $k$, then 
either       $ -p - 1 -1$ or $-p$  is multiple of $|G|$.
The first possibility implies that $|G|$ divides 4 which contradicts 
the assumption that $|G| \ge 5$.  The second possibility cannot occur 
because $G$ is not cyclic and hence it does not have order $p$. So we are done.
  
Finally, assume that $G$ is not cyclic and has order at most $4$. 
This means $G$ is the Klein four group $V_4$. We can take $M$ to be 
any indecomposable module of dimension $2n$ for $n>2$ (see, for 
example \cite{HR}).  The syzygies of the trivial $kV_4$-module 
are all odd dimensional, so $M$ is not one of them.  
In this case, proper subgroups are cyclic and
there are only 4 isomorphism classes of module induced from proper subgroups,
and they all have dimension 2 or 4 and our module has 
dimension $2n$ where $n > 2$. 
\end{proof}

\begin{rem}
  The reader might note that in \cite{CarCheMin2},
  the main theorem characterizers
the groups in which all ghosts between modules in the thick subcategory 
of $\stmod(kG)$ generated by $k$ are zero. This restriction to this 
subcategory is not necessary
for strong ghosts, because as noted in Proposition \ref{strong-restrict},
the property of being a strong ghost is detected by restrictions to 
$p$-subgroups, and for a $p$-group, the stable category is generated by 
the trivial module. 
\end{rem}

\section{Cohomology in a bounded range determines ghosts} \label{boundedrange}

Our first proposition shows that in order to verify that a map is a ghost 
it is enough to check the induced map in cohomology in finitely many degrees.

\begin{prop} \label{prop:boundghost}
Let $M$ and $N$ be two finitely generated $kG$-modules.
 There exists a nonnegative integer $d$  such that  
 if $\varphi \colon M \rar N$ is a $kG$-homomorphism with the property that
\[
\xymatrix{
  \HHHH^i(G, \varphi) \colon \HHHH^i(G, M) \ar[r] & \HHHH^i(G, N)
}
\]
is zero for all $-d \le i \le d$, then $\varphi$ is a ghost.
\end{prop}

\begin{proof}
Let $\mathcal{G}(M, N)$ denote the set of all ghost maps from $M$ to $N$. 
Let
\[
S_i \ = \ \{\varphi: M \to N \ \vert \ 
\HHHH^j(G,\varphi) =0 \quad \text{for all $j$ such that} 
\quad -i \leq j \leq i \}.  
\]
Consider the descending sequence of subspaces
\[
S_0 \ \supseteq S_1 \ \supseteq \  S_{2} \ \supseteq \ \cdots
\]
This sequence stabilizes because all the subspaces live in a
finite dimensional vector space. So there exists an integer $d$
such that $S_d = S_{d+1} = \dots = \mathcal{G}(M,N)$. 
This equation is equivalent to the assertion in the the proposition.
\end{proof} 

Note that the integer $d$ in the statement of this proposition
depends only on $M$ and $N$ and not on
the map $\varphi$ between them. 


We now prove a duality result which is used to show that ghosts can 
be detected in bounded non-negative degrees.  
Let $M^* = \Hom_k(M,k)$ be the $k$-dual of a $kG$-module $M$. If
$\varphi: M \to N$ is a $kG$-homomorphism, then the naturality of the
functor $\Hom_k$ yields a map $\varphi^*: N^* \to M^*$.

\begin{prop} \label{prop:dualghost}
  The dual of a ghost is a ghost, and the dual of a strong ghost is a
  strong ghost.
\end{prop}

\begin{proof}
Suppose that $\varphi: M \to N$ is a ghost.
Recall that Tate duality gives a natural isomorphism
\[\HHHH^{-i -1} (G, L) \cong (\HHHH^i (G, L^*))^*\]
for any finite-dimensional module $L$. Thus, for each $i$, we have the
following commutative diagram, where the vertical maps are induced by $\varphi$:
\[
\xymatrix{
 \HHHH^{-i -1}(G, M) \ar[r]^{\cong} \ar[d] & (\HHHH^i(G, M^*))^* \ar[d] \\
  \HHHH^{-i -1}(G, N) \ar[r]^{\cong} & (\HHHH^i(G, N^*))^*.
}
\]
 Since the two horizontal maps are isomorphisms, the right vertical map is
 zero because the left vertical map is zero. Consequently, $\varphi^*$ is
 also a ghost. The statement about strong ghost follows from the
 fact that the dual operation and Tate duality commute with restriction
 to a subgroup. 
\end{proof}

The next result is a corollary of the above proof. The point is that the 
second condition in the corollary is equivalent (by the previous diagram)
to the statement that $\HHHH^i(G,\varphi)$ is zero for all $i \leq 0$. 

\begin{cor} \label{cor:ghostbycohomology}
 A map $\varphi \colon M \rar N$ between finitely generated
$kG$-modules is a ghost if and only if the following two conditions hold.
\begin{enumerate}
\item $\HHHH^i(G, \varphi) \colon \HHHH^i(G, M) 
\rar \HHHH^i(G, N)$ is zero for all $i \ge 0$.
\item $\HHHH^i(G, \varphi^*) \colon \HHHH^i(G, N^*) 
\rar \HHHH^i(G, M^*)$ is zero for all $i \ge 0$.
\end{enumerate}
\end{cor}

Recall that the Evens-Venkov Theorem states that for any
finitely generated $kG$-module $M$, the ordinary cohomology 
$\HHH^*(G, M)$ is finitely generated as a module over $\HHH^*(G, k)$.
Moreover, the ring $\HHH^*(G, k)$ is a finitely generated 
$k$-algebra. This can be used to show that ghosts are detected
on non-negative cohomology. 

\begin{thm} \label{thm:ghostpositivecoho} 
  Let $M$ and $N$ be finitely generated $kG$-modules. Let positive
  integers  $m$ and $n$ be the least upper bounds for the 
degrees of the generators of  $H^*(G, M)$ and $H^*(G, N^*)$ respectively.
If $\varphi \colon M \rar N$ is any map such that
\[
\HHHH^i(G, \varphi) = 0 \ \  \text{for } \  0 \le i \le  m \ \  
\text{and}  \ \   \HHHH^i(G, \varphi^*) = 0 \ \  \text{for} \ \
0 \le i \le  n,
\]
then $\varphi$ is a ghost. 
\end{thm}

\begin{proof}
  By Corollary \ref{cor:ghostbycohomology}, to show that $\varphi$
  is a ghost, it is enough to show that 
  $\HHHH^i(G , \varphi)$ and $\HHHH^i(G, \varphi^*)$ are both zero maps
  for all $i \ge 0$. 
  Since $\HHHH^i(G, \varphi)$ is zero for all $i$ with $0 \leq i \le m$, 
it is zero in all of 
  the degree where the generators are located. Thus, it is zero
  in all nonnegative degrees. The same also holds for $H^i(G, \varphi^*)$.
\end{proof}

\section{Eventual ghosts and groups with periodic cohomology}
\label{eventualghosts}

We say that a map $\varphi \colon M \rar N$ between finitely generated
$kG$-modules is an \emph{eventual ghost} if there exists
an integer $n$ such that $\HHHH^i(G, \varphi) = 0$  for all $i \ge n$.
Clearly, every ghost is also an eventual ghost. 
In this section we show that the converse holds if and only
if $G$ has periodic cohomology.   We begin with a lemma which
gives a sufficient condition for eventual ghosts.

\begin{lemma} \label{mainlemma}
  Let $M$ be a finitely generated $kG$-module. Assume that $\HHH^*(G,k)$
  is generated in degrees at most $d$ and that $\HHH^*(G,M)$ is generated
  as a right $\HHH^*(G,k)$-module in degrees at most $m$. Let $\varphi:M \to N$
  be a homomorphism, and suppose that for some $t>m$,
  $\HHH^i(G, \varphi) = 0$ for all $i$ such that $t+1 \leq i \leq t+d$.
  Then $\HHH^i(G, \varphi)$ is zero for all $i \ge t + 1$.
  
\end{lemma}

\begin{proof}
  Let $\{ \zeta_1, \zeta_2, \cdots \zeta_u\}$ be a set of generators in positive degrees for
  $\HHH^*(G,k)$, and let $d_i\leq d$ denote the degree of $\zeta_i$.  Then
  for any $n$ with $n >m$, 
\[
\HHH^n(G, M) = \sum_i \HHH^{n-d_i}(G, M) \zeta_i.
\]
Taking the induced map in cohomology, we have that 
\[
\HHH^n(G, \varphi) = \sum_i \HHH^{n-d_i}(G, \varphi) \zeta_i
\]
By hypothesis $\HHH^{i}(G, \varphi) = 0$
for $t+1 \leq i \leq t+d$. Inductively, assume that
$\HHH^{i}(G, \varphi) = 0$ for $t+1 \leq i < n$ and $n > t+d$.
Then by the last equation,  $\HHH^{n}(G, \varphi) = 0$, since
$n-d_i \geq n-d \ge t+1$. Thus, the proof follows by induction. 
\end{proof}

Let $G$ be a finite group and let $k$ be a field of characteristic $p$.
A group $G$ is said to have periodic cohomology if there exists a class
$\eta$ in $\HHH^d(G, k)$ such that for $i \ge 0$ multiplication by $\eta$
gives an isomorphism
\[
\HHH^i(G, k) \cong
\HHH^{i+d}(G, k).
\]
Groups with periodic cohomology play an important role in
representation theory and topology. It is well known that
$G$ has periodic cohomology
if and only if the Sylow $p$-subgroup of $G$  is a cyclic
group or a generalized quaternion group.  In \cite{CarCheMin}
we proved that for every finitely generated $kG$-module $M$,
the Tate cohomology $\HHHH^*(G, M)$ is finitely generated as
a graded module over $\HHHH^*(G, k)$ if and only if $G$ has
periodic cohomology.

\begin{thm} \label{thm:periodic}
  Let $M$ be a finitely generated $kG$-module.
  If every eventual ghost map from $M$ is a ghost,
then $\HHHH^*(G, M)$ is a finitely generated module over $\HHHH^*(G, k)$.
\end{thm}

\begin{proof}
Suppose that $\HHH^*(G,k)$ is generated in degrees at most $d$ and that
$\HHH^*(G, M)$ as a module over 
$\HHH^*(G, k)$ is generated in degrees at most $m$. 
Choose a $k$-basis $\{ \theta_j\}$ 
for the finite-dimensional space $V = \sum_{i = m+1}^{m+d}H^i(G, M)$.
Each is represented by a cocycle $\theta_j : \Omega^{e_j}(k) \to M$, where
$e_j$ is the degree. We assemble them to form a map
$\eta: \sum \Omega^{e_j}(k) \to M$,
which is completed  to a  triangle in $\stmod(kG)$
\[
\xymatrix{
\oplus_j\,\Omega^{e_j}(k) \ar[r]^{\qquad \eta} & M  
\ar[r]^\varphi & L.
}
\]
Because the $\theta_j$'s generate $V$, the map $\varphi$ has the
property that $\HHH^i(G, \varphi) = 0$ for $m+1 \leq i \leq m+d$.
Thus, by Lemma \ref{mainlemma}, $\varphi$ 
is an eventual ghost, and hence a ghost.

Let $\gamma$ be an arbitrary homogeneous element 
in $\HHHH^*(G, M)$ in degree $t$. In the diagram
\[
\xymatrix{
\oplus_j\,\Omega^{e_j} (k)  \ar[rr]^{\eta} 
& &M \ar[rr]^{\varphi} & & L  \\
&& \Omega^t(k) \ar[u]^{\gamma} \ar[urr]_{0}\ar@{..>}[ull] &&
}
\]
$\varphi$ is a ghost and $\varphi \gamma$ is zero. 
Hence $\gamma$ factors
through $\eta$. This shows that 
that the classes $\{ \theta_j \}$ 
generate $\HHHH^*(G, M)$ as a module over $\HHHH^*(G, k)$.
\end{proof}

The next example shows that the converse of this theorem is not true.

\begin{example}   Let $G = C_2 \times C_2$. Consider 
  the generator $\eta$ of $\HHHH^{-1}(G, k)$ which is the Tate dual
  of the identity in $\HHHH^0(G,k)$. This can be represented as 
$\eta \colon \Omega^{-1}(k) \rar k.$
The domain of $\eta$ is $\Omega^{-1}(k)$, whose Tate cohomology is 
just a suspension of the Tate cohomology ring $\HHHH^{^{*}}(G, k)$. 
In particular it is finitely generated over $\HHHH^*(G, k)$.
Thus, $\eta$ is  an eventual ghost but not a ghost.  In fact, it follows 
from the multiplicative structure of the Tate cohomology ring of the
Klein group that $\HHHH^i(G, \eta)$ is nonzero only in 
degree $0$.
\end{example}

\begin{thm}
  Let $G$ be a finite group. Then every eventual ghost 
map in $\stmod(kG)$ is a ghost map if and only if $G$ has periodic cohomology. 
\end{thm}

\begin{proof}
The ``only if'' part is clear because if $G$ has periodic cohomology, 
say of period $d$, then we can pick $d$ consecutive
integers sufficiently large where the induced map in Tate 
cohomology is zero. But then periodicity of Tate
cohomology implies that they induce the zero map in Tate 
cohomology in all degrees.

If every eventual map is a ghost map, then the above theorem 
tells us that every finitely generated $kG$-module has finitely  
generated Tate cohomology.
By \cite[Theorem 4.1]{CarCheMin}, $G$ has periodic cohomology.
\end{proof}

Note that in the case when $G$ does not have periodic cohomology, 
the above theorem helps us construct an eventual-ghost map between 
$kG$-modules that is not a ghost.

\bibliographystyle{plain}

\end{document}